\documentclass{amsart}
\usepackage{amssymb,amsfonts,amsmath,amsthm}
\usepackage[all]{xy}
\usepackage[shortlabels]{enumitem}

\newtheorem{theorem}{Theorem}[section]

\newtheorem{proposition}{Proposition}[section]
\newtheorem{corollary}{Corollary}[section]
\newtheorem{definition}{Definition}[section]
\newtheorem{conjecture}{Conjecture}
\newtheorem{lemma}{Lemma}[section]

\theoremstyle{remark}
\newtheorem{remark}{Remark}

\newcommand{\n}{\nabla}
\newcommand{\fr}{\frac}
\newcommand{\lf}{\left}
\newcommand{\rg}{\right}
\newcommand{\al}{\alpha}
\newcommand{\real}{\mathbb{R}}

\begin{document}

\title[Generalizing axioms of $r$-planes and $r$-spheres] {Generalizing axioms of $r$-planes and $r$-spheres on Riemannian and K\"ahler manifolds}
\author{Cristina Levina}
\address{Departamento de Matem\'atica e Estat\'\i stica, Av. Pasteur, 458, Urca, Rio de Janeiro, CEP 22290-240, Brasil} \email{cristina.marques@uniriotec.br}
\author{S\'ergio Mendon\c ca}
\address{Departamento de An\'alise, Instituto de Matem\'atica,
Universidade Federal Fluminense, Niter\'oi, RJ, CEP 24020-140,
Brasil} \email{sergiomendonca@id.uff.br}

\thanks{This work was partially supported by CNPq, Brasil.}
\date{}



\begin{abstract}The famous theorems of Cartan, related to the axiom of $r$-planes, and Leung-Nomizu about the 
axiom of $r$-spheres were extended to K\"ahler geometry by several authors. 
 In this paper we replace the strong notions of totally geodesic submanifolds
($r$-planes) and extrinsic spheres ($r$-spheres) by a wider class of special isometric immersions such that theorems of type \lq\lq axioms of $r$-special
submanifolds\rq\rq\ could hold. We verify also that there are plenty of special submanifolds
in real and complex space forms and, in the codimension one case, in
Einstein manifolds. In the proof of our theorem in the complex case, a new class of K\"ahler manifolds arose naturally, which
we named $XY$-manifolds. They satisfy the symmetry property
$\lf<\bar R(X,JX)Y,JX\rg>=\lf<\bar R(Y,JY)X,JY\rg>$, where $J$ is the almost complex structure, 
$\bar R$ is the curvature tensor and $X, JX,Y,JY$ are orthonormal tangent
vectors.
\end{abstract}

\maketitle

\section{Introduction}

An $m$-dimensional Riemannian manifold $M$ is said to satisfy the
axiom of $r$-planes if there exists $2\le r\le m-1$ such that, for
any $p\in M$ and any $r$-dimensional linear subspace $W$ of the
tangent space $T_pM$, there exists a totally geodesic submanifold $S$
with $T_pS=W$. Cartan proved that any manifold satisfying this axiom
has constant sectional curvature (\cite{Ca}).

A umbilical submanifold of a Riemannian manifold $M$ is called an
extrinsic sphere when it has parallel mean curvature vector. One
says that an $m$-dimensional Riemannian manifold $M$ satisfies the
axiom of $r$-spheres if there exists $2\le r\le m-1$ such that for
any $p\in M$ and any $r$-dimensional linear subspace $W$ of $T_pM$
there exists an extrinsic sphere $S$ with $T_pS=W$. Leung and Nomizu
extended Cartan's theorem proving that if $M$ satisfies the axiom of
$r$-spheres then it has constant sectional curvature (\cite{LN}).

A Hermitian manifold $M$ of real dimension $2m$ is said to satisfy the axiom of holomorphic 
$2r$-spheres if there exists $1\le r\le m-1$ such that for
any $p\in M$ and any $2r$-dimensional holomorphic subspace $W$ of $T_pM$
there exists an extrinsic sphere $S$ satisfying $T_pS=W$. Similarly, 
$M$ satisfies  the axiom of antiholomorphic
$r$-spheres if there exists $2\le r\le m-1$ such that for
any $p\in M$ and any $r$-dimensional antiholomorphic subspace $W$ of $T_pM$
there exists an extrinsic sphere $S$ satisfying $T_pS=W$. If $M$ is a 
K\"ahler manifold, any of these two 
hypotheses imply that $M$ has constant holomorphic curvature (\cite{co}, \cite{g}, \cite{gm}, \cite{h}, \cite{my}).

This paper was motivated by the following question: what could
be the most general extension of the notions of totally geodesic
submanifolds and extrinsic spheres so that results similar to the
above theorems would still hold? 
First one could ask if some theorem envolving an axiom of
$r$-minimal submanifolds could be true. However, we expect that the answer
for this question is negative. We even conjecture that, given
any point $p$ in a Riemannian manifold $M$, any linear subspace
$W\subset T_pM$ and any vector $v$ orthogonal to $W$, there exists a 
submanifold $S$ tangent to $W$ at $p$ having parallel mean curvature vector $H$  
with $H(p)=v$. In the case $v=0$ this could be called an \lq\lq
infinitesimal Dirichlet problem\rq\rq.

We will define a class of isometric immersions which
contains the totally geodesic submanifolds and the 
extrinsic spheres.
 Such a class satisfies a theorem of type \lq\lq
axiom of $r$-special immersions\rq\rq. 

We first fix some notations and definitions. Fix an isometric immersion $f: S\to M$
and $p\in S$. By $(T_pS)^\perp$ we denote the orthogonal complement of $T_pS$ 
relatively to $T_pM$. If $Z\in T_pM$ we denote by $Z^\perp$ the projection 
of $Z$ onto $(T_pS)^\perp$. If $\bar \n$ is the 
Levi-Civita connection of $M$, $X$ is a vector field on $S$ and $\eta$ is 
a vector field on $M$ which is orthogonal to $S$, we set $\n^\perp_X\eta=(\bar\n_X\eta)^\perp$. 
Let $\al$ denote the second fundamental form of $f$, namely, we write $\al(X,Y)=(\bar\n_XY)^\perp$, where $X, Y$ are vector fields on $S$.   
We recall that an orthonormal frame $(X_i)$ on a neighborhood of $p$
in $S$ is said to be geodesic at $p$
if $\lf(\n_{X_i}X_j\rg)(p)=0$
for all $i,j$, where $\n$ is the Levi-Civita connection of $S$.

\begin{definition} Let $f:S\to M$ be an isometric immersion. We say that $f$ is special if, for any $p\in S$ and any orthonormal basis $v_1,\cdots,v_r$ of $T_pS$, it holds that
\begin{equation}\label{R}\sum_i\bar R(v_i,w)v_i \in T_pS,\end{equation}
for all $w\in T_pS$, where $\bar R(X, Y) =\bar \n_X\bar \n_Y-\bar \n_Y\bar \n_X-\bar \n_{[X,Y]}$ is the curvature tensor of $M$. 
\end{definition}
\begin{remark} We will see in Lemma \ref{independence} below that Equation (\ref{R}) does not depend 
on the choice of the orthonormal basis $v_1,\cdots,v_r$ of $T_pS$. 
\end{remark}

Now we present a stronger definition.
\begin{definition} \label{defvery}  Let $f:S\to M$ be an isometric immersion where $S$ is an $r$-dimensional manifold.
We say that $f$ is very special if for any $p\in S$ and any vectors $v, w\in T_pS$, it holds that
\begin{equation}\label{veryR}\bar R(v,w)v \in T_pS.\end{equation} 
\end{definition}

\begin{remark} It follows from the definition that any immersed curve is very special. 
\end{remark}

\begin{remark} We will see in Proposition \ref{umbilical} below that, for dimensions greater than $1$, umbilical submanifolds are 
very special if, and only if, they are extrinsic spheres, hence the class of very special isometric immersions contains 
the extrinsic spheres. Note that this is false in the case that the dimension of $S$ is $1$, since 
in this situation all immersions are very special 
and umbilical, but not necessarily extrinsic spheres.
\end{remark}

\begin{definition} Given a Riemannian manifold $M$ of dimension $m\ge 3$ and $2\le r\le m-1$, we will say that 
$M$ satisfies the axiom of special (respectively, very special) $r$-submanifolds if, for 
any $p\in M$ and any $r$-dimensional linear subspace $W$ of $T_pM$, there exists a special (respectively, very special) 
submanifold $S$ satisfying $T_pS=W$. 
\end{definition}

\begin{definition}
 Let $M$ be a K\"ahler manifold of real dimension $2m,\ m\ge 2$. We define similar axioms in complex geometry:
\begin{enumerate}[(a)]
\item Given  
$1\le r\le m-1$, we say that $M$ satisfies the axiom of special holomorphic $2r$-submanifolds if, for 
any $p\in M$ and any $2r$-dimensional holomorphic linear subspace $W$ of $T_pM$, there exists a special 
submanifold $S$ satisfying $T_pS=W$;
\item Given $2\le r\le m$, we say that $M$ satisfies the 
axiom of special antiholomorphic $r$-submanifolds if, for 
any $p\in M$ and any $r$-dimensional antiholomorphic linear subspace $W$ of $T_pM$, there exists a special  
submanifold $S$ satisfying $T_pS=W$.
\end{enumerate}
\end{definition}

Our first result is the following extension of
the Theorems of Cartan and Leung-Nomizu.
\begin{theorem} \label{riemcase} Let $M$ be an $m$-dimensional Riemannian manifold with 
$m\ge 3$. Then it holds that:
\begin{enumerate}[(a)]
\item \label{realspecial}If $M$ satisfies the axiom of special $r$-submanifolds for some $2\le r\le m-1$, then $M$ has 
constant sectional curvature if $r\le m-2$ and $M$ is Einstein if $r=m-1$; 
\item If $M$ satisfies the axiom of very special $(m-1)$-submanifolds, then $M$ has 
constant sectional curvature.
\end{enumerate}
Reciprocly, 
any isometric immersion in a manifold of constant sectional curvature is very special 
and any hypersurface in an Einstein manifold is
special.
\end{theorem}

\begin{remark} Note that Einstein manifolds show that Item {\it\ref{realspecial}} in Theorem \ref{riemcase} may not be improved to
obtain constant sectional curvature in the case $r=m-1$.
\end{remark}

 Before stating our similar result in K\"ahler geometry (Theorem \ref{complextheorem} below) we will introduce the 
 class of $XY$-manifolds, which has an important tool in the proof of this theorem.  
\begin{definition} A K\"ahler manifold $M$ of real dimension $2m,\ m\ge 2,$ 
with almost complex structure $J$ is said to be a $XY$-manifold if for any 
local orthonormal vector fields $X, JX, Y, JY$ it holds that 
\begin{equation}\label{xyequation}\lf<\bar R(X, JX) Y, JX\rg>=\lf<\bar R(Y,JY)X,JY\rg>.\end{equation}
\end{definition}
It is well known that the assumption that $\lf<\bar R(X, JX) Y, JX\rg>=0$ is equivalent to the fact that $M$ has constant holomorphic 
sectional curvature (see Proposition \ref{equal0} below). Thus the class of $XY$-manifolds includes  
K\"ahler manifolds with constant holomorphic sectional curvature. We would like to propose the following
\begin{conjecture} There exist examples of  $XY$-manifolds with non-constant holomorphic sectional curvature. 
\end{conjecture}
 
In K\"ahler geometry we present the following extension of the results in \cite{co}, \cite{g}, \cite{gm}, \cite{h}, \cite{my}.
\begin{theorem} \label{complextheorem}
 Let $M$ be a K\"ahler manifold of real 
dimension $2m,\ m\ge 2$. Then $M$ has constant holomorphic sectional curvature if 
one of the following conditions hold:
\begin{enumerate}[(a)]
\item $M$ satisfies the axiom of special holomorphic $2r$-submanifolds, for some $1\le r\le m-1$;
\item $M$ satisfies the axiom of special antiholomorphic $r$-submanifolds, for some $2\le r\le m$.
\end{enumerate}
Reciprocly, all complex and all totally real immersions in a K\"ahler manifold of constant holomorphic sectional curvature are very special. 
\end{theorem}

\section{Some examples} 

Since any curve is very special, if we take immersed curves $f_i:(a_i,b_i)\to M_i$, $i=1,\cdots, n$, 
then $F:(a_1,b_1)\times\cdots\times (a_n,b_n)\to M_1\times\cdots\times M_n$, given 
by $F(t_1,\cdots,t_n)=\bigl(f_1(t_1),\cdots,f_n(t_n)\bigr)$ is 
a very special immersion.
Now we list some other examples: any isometric immersion in a manifold of constant sectional curvature is very special, and each  hypersurface of an Einstein manifold is special (see Theorem \ref{riemcase} above); 
any complex, or totally real, immersion in a K\"ahler manifold with constant holomorphic 
curvature is very special (see Theorem \ref{complextheorem} above). 
Riemannian products of the above examples provide other special (or very special) immersions.

\section{\label {preliminaries}Preliminaries}

We begin this section proving the following lemma.

\begin{lemma} \label{equivalence}
Let $f:S\to M$ be an isometric immersion where $S$ is an $r$-dimensional manifold.
Then $f$ is special if, and
only if, for any
$p\in S$ and any local orthonormal frame $(X_i)$ on $S$ which is geodesic at $p$
we have:
\begin{equation}\label{definition} r\lf(\n_{X_j}^\perp H\rg)(p)=\lf(\sum_{i=1}^r\n_{X_i}^\perp
\al(X_i,X_j)\rg)(p), \text{\ for any\ }j\, ,\end{equation}
where $r$ is the dimension of $S$ and $H=\fr 1r\sum_{i=1}^r\al(X_i,X_i)$ is 
the mean curvature vector. 
\end{lemma}
\begin{proof} For local vector fields $X,Y,Z$ on $S$ the Codazzi equation
says that
\begin{equation}(\bar
R(X,Y)Z)^\perp=(\nabla_X^\perp\alpha)(Y,Z)-(\nabla_Y^\perp\alpha)(X,Z),\end{equation}
where
$$(\nabla_X^\perp\alpha)(Y,Z)=\nabla_X^\perp\alpha(Y,Z)-\alpha(\n_XY,Z)-\al(Y,\n_XZ).$$

Fix $\eta\perp T_pS$ and some local vector field $N$ orthogonal to $S$ satisfying $N(p)=\eta$. Take orthonormal vectors $v_1,\cdots,v_r\in T_pS$ and some extension of them to a local orthonormal
frame $X_1,\cdots,X_r$ on $S$ which is geodesic at $p$. Fix $i,j$. At the point $p$ we have:
\begin{eqnarray*}\lf<\bar
R(v_i,v_j)v_i,\eta\rg>&=&\lf<\bar
R(X_i,X_j)X_i,N\rg>(p)\\ &=&\lf<(\nabla_{X_i}^\perp\alpha)(X_j,X_i)-(\nabla_{X_j}^\perp\alpha)(X_i,X_i),N\rg>(p)\\
&=&\lf<\nabla_{X_i}^\perp\alpha(X_j,X_i)-\nabla_{X_j}^\perp\alpha(X_i,X_i),N\rg>(p),
\end{eqnarray*}
where we used the fact that
$\lf(\n_{X_i}X_j\rg)(p)=\lf(\n_{X_i}X_i\rg)(p)=\lf(\n_{X_j}X_i\rg)(p)=0$. By summing up we obtain that
\begin{equation} \label{nabla}\lf<\sum_i\bar
R(v_i,v_j)v_i-\lf(\sum_i\nabla_{X_i}^\perp\alpha(X_j,X_i)\rg)(p)+r\lf(\n_{X_j}^\perp
H\rg)(p)\, ,\,\eta\rg>=0.\end{equation} From (\ref{nabla})
we obtain easily that $f$ is special if, and only if, $$r\lf(\n_{X_j}^\perp H\rg)(p)=
\lf(\sum_{i=1}^r\n_{X_i}^\perp
\al(X_i,X_j)\rg)(p).$$ Lemma \ref{equivalence} is proved.\end{proof}

\begin{lemma}\label{independence} Equation  {\rm(\ref{R})} does 
not depend on the choice of an orthonormal basis $v_1,\cdots,v_r$ of $T_pS$. Similarly 
 {\rm(\ref{definition})} does not depend on
the choice of a local orthonormal frame $X_1,\cdots,X_r$ which is geodesic at $p$. 
 \end{lemma} 
 \begin{proof} Fix an orthonormal basis $(v_i)$ of $T_pS$, $w\in T_pS$ and $\eta\in (T_pS)^\perp$. Now consider the linear map 
$A_{w\eta}:T_pS\to T_pS$ given by $A_{w\eta}(v)=\pi(\bar R(w,v)\eta)$, where $\pi:T_pM\to T_pS$ is 
the standard orthogonal projection.  For the trace ${\rm{tr}}(A_{w\eta})$ we 
have:
\begin{equation}\label{trace}{\rm{tr}}(A_{w\eta})=\sum_i\lf<\bar R(w,v_i)\eta,v_i\rg>=
\sum_i\lf<\bar R(v_i,w)v_i,\eta\rg>.
\end{equation}
From (\ref{trace}) we see that the choice of the orthonormal basis $(v_i)$ is irrelevant on (\ref{R}). From this 
and Lemma \ref{equivalence} we obtain that the choice of the orthonormal frame $(X_i)$ geodesic at $p$ is 
irrelevant in Equation (\ref{definition}). 
\end{proof}

{\rm By using a proof easier and similar to the proof of Lemma \ref{equivalence} we have the following}
\begin{lemma} \label{lemmavery} Let $f:S\to M$ be an isometric immersion. We have that $f$ is very special if, and only if, for any
$p\in S$ and any local orthonormal fields $X,Y$ on $S$ satisfying $\nabla_XY(p)=\nabla_YX(p)=\nabla_XX(p)=
\nabla_YY(p)=0$, 
we have:
\begin{equation}\label{very} \lf(\n_{Y}^\perp \al(X,X)\rg)(p)=\lf(\n_{X}^\perp
\al(X,Y)\rg)(p).
\end{equation}
\end{lemma}

\begin{proposition}\label{umbilical} Let $f:S\to M$ be a umbilical immersion and $r$ the dimension of $S$. 
If $r\ge 2$ then the  following conditions are equivalent:
\begin{enumerate} 
\item $f$ is special;
\item $f$ is very special;
\item $f$ is an extrinsic sphere.
\end{enumerate}
\end{proposition}
\begin{proof} Fix $p\in S$ and an orthonormal frame $(X_i)$ in a neighborhood of $p$  which 
is geodesic at $p$.  Since $f$ is umbilical we have that $\al(X,Y)=\lf<X,Y\rg>H$. Assume 
that $f$ is special and fix $j\in\{1,\cdots,r\}$. It follows from (\ref{definition}) that 
$$r\lf(\n_{X_j}^\perp H\rg)(p)\!=\!\lf(\sum_{i=1}^r\n_{X_i}^\perp
\al(X_i,X_j)\rg)(p)\!=\!\lf(\sum_{i=1}^r\n_{X_i}^\perp
(\delta_{ij}H)\rg)(p)\!=\!\lf(\n_{X_j}^\perp H\rg)(p).$$
Since $r\ge 2$ we have that $\n_{X_j}^\perp H=0$, hence $f$ is an extrinsic sphere. 

Now assume that $f$ is an extrinsic sphere. Fix $p\in S$ and local orthonormal fields $X,Y$ on $S$ satisfying $\nabla_XY(p)=\nabla_YX(p)=\nabla_XX(p)=
\nabla_YY(p)=0$. Then both sides in (\ref{very}) vanish, hence $f$ is very special. Proposition \ref 
{umbilical} is proved. 
\end{proof}

\section{The Riemannian case - proof of Theorem \ref{riemcase}}
We first assume that $M$ satisfies the axiom of special $r$-submanifolds for some 
$2\le r\le m-2$.  To show
that the sectional curvature is constant it is sufficient, by
Schur's Lemma (see for example \cite {S}, II p. 328), to prove that at any point
$p\in M$ the sectional curvature is constant for planes contained in
$T_pM$. To show this last fact it suffices to obtain that $\lf<\bar
R(v,w)v,\eta\rg>=0$ for all orthonormal vectors $v,w,\eta$ (see for
example Lemma 1.9 in \cite{D}).

So we fix $p\in M$ and orthonormal vectors $v,w,\eta\in T_pM$. Since $2\le r\le m-2$,
we can construct an orthonormal set $$\{v_1=v,\
v_2=w,\,v_3,\cdots,v_{r+1},\eta\}\subset T_pM.$$ 
By our hypothesis there exists a special submanifold $S_1$ containing $p$ such that
$$T_p(S_1)=\text{\rm span}(v_1,\cdots,v_r),$$ 
where $\text{\rm span}(v_1,\cdots,v_r)$ denotes the linear subspace
generated by $v_1,\cdots,v_r$. By the definition of a special submanifold we have that
\begin{equation}\label{basicsum}\sum_{i=1}^r\lf<\bar R(v_i,w)v_i,\eta\rg>=0.\end{equation}
Note that $\bar R(v_2,w)v_2=0$, since $w=v_2$. Fix $1\le j\le r$, $j\not=2$. By
hypothesis there exists a special submanifold $S_2$ containing $p$ such that
\begin{equation}T_p({S_2})=\text{\rm
span}(v_1,\cdots,\hat v_{j},\cdots,v_r,v_{r+1}),\end{equation}
where  \ $\hat v_j$ means that the vector $v_j$ is omitted. Since we have that $w=v_2\in \text{\rm
span}(v_1,\cdots,\hat v_{j},\cdots,v_r,v_{r+1})$, the definition of a special 
submanifold implies that
\begin{equation}\label{secondsum}\lf({\sum_{1\le i\le r,\, i\not=j}}\lf<\bar R(v_i,w)v_i,\eta\rg>\rg)+\lf<\bar R(v_{r+1},w)v_{r+1},\eta\rg>=0.\end{equation}
From (\ref{basicsum}) and (\ref{secondsum}) we obtain that $\lf<\bar R(v_j,w)v_j,\eta\rg>=\lf<\bar
R(v_{r+1},w)v_{r+1},\eta\rg>$, for all $1\le j\le r,\, j\not=2$. Thus we obtain that
\begin{equation}0=\lf<\sum_{i=1}^r\bar R(v_i,w)v_i,\eta\rg>=(r-1)\lf<\bar
R(v_1,w)v_1,\eta\rg>=(r-1)\lf<\bar
R(v,w)v,\eta\rg>,\end{equation} hence $\lf<\bar R(v,w)v,\eta\rg>=0$ and $M$ has constant sectional curvature.

Now we assume that $M$ satisfies the axiom of special $(m-1)$-submanifolds. To show that $M$ is
Einstein, it suffices to prove that, for any point $p\in M$ and any
orthonormal vectors $w,\eta\in T_pM$, the Ricci bilinear form satisfies
\begin{equation}\label{ricci}-{\rm Ric}(w,\eta)=\lf<\sum_{i=1}^m\bar
R(v_i,w)v_i,\eta\rg>=0,\end{equation} for some orthonormal basis $v_1,\cdots,v_m$
of $T_pM$. Without loss of generality we may assume that $v_m=\eta$,
hence $w\in\text{\rm span}(v_1,\cdots,v_{m-1})$. By hypothesis there
exists a special submanifold $S$ containing $p$ such that $T_pS=\text{\rm
span}(v_1,\cdots,v_{m-1})$, hence we obtain that
$\sum_{i=1}^{m-1}\bar R(v_i,w)v_i\in T_pS$, hence (\ref{ricci}) holds and $M$ is Einstein.

To finish the proof of Theorem \ref{riemcase} we need to show that 
every isometric immersion in a manifold $M$ of constant sectional curvature is very special,
and that every codimension one isometric immersion in an Einstein
manifold is special. 

First we consider the case 
that $M$ has constant sectional curvature $\rho$. Consider an isometric immersion 
$f:S\to M$. Fix $p\in S$. Given vectors $v, w\in T_pS$, 
we have that $\lf<\bar
R(v,w)v,\eta\rg>=\rho(\lf<v,\eta\rg>\lf<w,v\rg>-\lf<v,v\rg>\lf<w,\eta\rg>)=0$ for all $\eta\in (T_pS)^\perp$, hence $\bar R(v,w)v\in T_pS$. As a consequence $f$ is very special. 

Now we consider the case that $M$ is Einstein. Let $f:S\to M$ be a
codimension one isometric immersion. Fix $p\in S$ and $\eta\in(T_pS)^\perp, |\eta|=1$. Consider
orthonormal vectors $v_1,\cdots, v_{m-1}\in T_pS$. Since $M$ is Einstein we 
have that Ric$(v,w)=\rho\lf<v,w\rg>$ for some $\rho\in \real$. Set $v_m=\eta$.  So we have
$$\sum_{i=1}^{m-1}\lf<\bar
R(v_i,w)v_i,\eta\rg>=\sum_{i=1}^{m}\lf<\bar
R(v_i,w)v_i,\eta\rg>=-\rho\lf<w,\eta\rg>=0,$$ for all $w\in T_pS$. We conclude that $f$ is special.  Theorem \ref{riemcase} is proved.

\section{Properties of $XY$-manifolds}
The following two propositions will be needed in the proof of Theorem \ref{complextheorem} 
\begin{proposition}\label{xyequivalence} Let $M$ be a K\"ahler manifold of real dimension $2m,\ m\ge 2$, 
with almost complex structure $J$. Fix $p\in M$.
The following 
assertions are equivalent:
\begin{enumerate} [(a)]
\item \label{xy} For any orthonormal vectors $X, JX, Y, JY \in T_pM$ it holds that $$\lf<\bar R(X,JX)Y,JX\rg>=\lf<\bar R(Y,JY)X,JY\rg>;$$
\item \label{-xy} For any orthonormal vectors $X, JX, Y, JY \in T_pM$ it holds that $$\lf<\bar R(X,JX)Y,X\rg>=-\lf<\bar R(Y,JY)X,Y\rg>;$$
\item \label{xy0} For any orthonormal vectors $X, JX, Y, JY \in T_pM$ it holds that $$\lf<\bar R(X,Y)X,JY\rg>=0.$$
\end{enumerate}
\end{proposition}
\begin{proof} If we replace $Y$ by $JY$ in \ref{xy} we obtain \ref{-xy}. If we replace $Y$ by $-JY$ in \ref{-xy} we 
obtain \ref{xy}.

Now we will prove that \ref{-xy} and \ref{xy0} are equivalent. We set $A=\frac{X+Y}{\sqrt{2}}, B=\frac{X-Y}{\sqrt{2}}$, 
hence $X=\frac{A+B}{\sqrt{2}}, Y=\frac{A-B}{\sqrt{2}}$. Clearly $A, JA, B, JB$ are orthonormal if and only if $X, JX, Y, JY$ 
are orthonormal. We have:
\begin{eqnarray*}\label{proofequivalence}
4\lf<\bar R(X,Y)X,JY\rg>&=&\lf<\bar R(A+B,A-B)(A+B),\,
JA-JB\rg>\\
&=&
2\lf<\bar R(B,A)(A+B),\,
{JA-JB}\rg>\\ 
&=&2\lf\{\lf<\bar R(B,A)A,JA\rg>-\lf<\bar R(B,A)B,JB\rg>\rg\}\\
&=&2\lf\{\lf<\bar R(A,JA)B,A\rg>+\lf<\bar R(B,JB)A,B\rg>\rg\},
\end{eqnarray*}
where we used the fact that $\lf<\bar R(B,A)A,JB\rg>=\lf<\bar R(B,A)B,JA\rg>$. From the above equation it is clear that \ref{-xy} is equivalent to \ref{xy0}.
\end{proof} 
\begin{proposition} \label{ZJZproperty}Let $M$ be a $XY$-manifold of real dimension $2m$ with $m\ge 3$. Fix 
$p\in T_pM$. Take orthonormal vectors $X, JX, Y, JY, Z, JZ\in T_pM$.  Then it holds that
\begin{eqnarray*} \lf<\bar R(X, JX)Y,JX\rg>&=&2\lf<\bar R(Z,JZ)X,JY\rg>\\
&=&4\lf<\bar R(Z,X)Z,Y\rg>=
4\lf<\bar R(JZ,X)JZ,Y\rg>.
\end{eqnarray*} 
\end{proposition}
\begin{remark} It is very surprising that the right hand side of the above equation does not depend on the choice of the unit vector 
$Z$ in ${\rm span}(X,JX, Y, JY)^\perp$.
\end{remark}
\begin{proof}[Proof of Proposition {\rm\ref{ZJZproperty}}] We will apply the definition of $XY$-manifolds to the orthonormal vectors $X,JX,\frac{Y+Z}{\sqrt{2}}, \frac{JY+JZ}{\sqrt{2}}$, obtaining that
\begin{equation*}\lf<\bar R(X, JX)\,\frac{Y+Z}{\sqrt{2}},\,JX\rg>=
\lf<\bar R\lf(\frac{Y+Z}{\sqrt{2}}, \frac{JY+JZ}{\sqrt{2}}\rg)X,\frac{JY+JZ}{\sqrt{2}}\rg>,
\end{equation*}
hence 
\begin{eqnarray*}2\lf<\bar R(X, JX)\,(Y+Z),\,JX\rg>&=&\lf<\bar R(Y, JY)X,JY\rg>+\lf<\bar R(Y, JY)X,JZ\rg>\\
&+&\lf<\bar R(Y, JZ)X,JY\rg>+\lf<\bar R(Y, JZ)X,JZ\rg>\\
&+&\lf<\bar R(Z, JY)X,JY\rg>+\lf<\bar R(Z, JY)X,JZ\rg>\\
&+&\lf<\bar R(Z, JZ)X,JY\rg>+\lf<\bar R(Z, JZ)X,JZ\rg>.
\end{eqnarray*}
By using again that $M$ is a $XY$-manifold we may cancel the first and last terms on the right side with corresponding terms on the left side, obtaining:
\begin{eqnarray*}\lf<\bar R(X, JX)\,(Y+Z),\,JX\rg>&=&\lf<\bar R(Y, JY)X,JZ\rg>+\lf<\bar R(Y, JZ)X,JY\rg>\\
&+&\lf<\bar R(Y, JZ)X,JZ\rg>+\lf<\bar R(Z, JY)X,JY\rg>\\
&+&\lf<\bar R(Z, JY)X,JZ\rg>+\lf<\bar R(Z, JZ)X,JY\rg>.
\end{eqnarray*}
Now we replace $Z$ by $-Z$ obtaining a new equality, which we add to the above equation and divide by $2$. All 
terms where $Z$ appears just one time will be cancelled. Thus we obtain:
\begin{equation}\label{almostdone}\lf<\bar R(X, JX)Y,JX\rg>=2\lf<\bar R(Y, JZ)X,JZ\rg>+\lf<\bar R(Z, JZ)X,JY\rg>,
\end{equation}
where we also used the fact that $\lf<\bar R(Y, JZ)X,JZ\rg>=\lf<\bar R(Z, JY)X,JZ\rg>$. Now we change 
$Z$ by $JZ$, obtaining that
\begin{equation}\label{decisive}\lf<\bar R(X, JX)Y,JX\rg>=2\lf<\bar R(Y, Z)X,Z\rg>+\lf<\bar R(Z, JZ)X,JY\rg>.
\end{equation}
\end{proof}
By (\ref{almostdone}) and (\ref{decisive}) we obtain that
\begin{equation}\label{ZZJZJZ} \lf<\bar R(Y, JZ)X,JZ\rg>=\lf<\bar R(Y, Z)X,Z\rg>.
\end{equation}
By (\ref{almostdone}) and (\ref{ZZJZJZ}) we have that
\begin{equation}\label{finally}\lf<\bar R(X, JX)Y,JX\rg>=\lf<\bar R(Y, Z)X,Z\rg>+\lf<\bar R(Y, JZ)X,JZ\rg>+\lf<\bar R(Z, JZ)X,JY\rg>.
\end{equation} 
By the Bianchi equality it is well known that 
\begin{equation}\label{Bianchi}\lf<\bar R(Z, JZ)X,JY\rg>=\lf<\bar R(Y, Z)X,Z\rg>+\lf<\bar R(Y, JZ)X,JZ\rg>.
\end{equation}
Thus Proposition \ref{ZJZproperty} follows directly from (\ref{ZZJZJZ}), (\ref{finally}) and (\ref{Bianchi}).
The following proposition will not be used in the proof of Theorem \ref{complextheorem}. 
\begin{proposition}Let $M$ be a $XY$-manifold of real dimension $2m, m\ge 4$. Then it holds that
$$\lf<\bar R(X,Y)Z,W\rg>=0,$$
for any orthonormal vectors $X,JX,Y,JY,Z,JZ,W,JW\in T_pM$ and any $p\in M$.
\end{proposition}
\begin{proof} Fix orthonormal vectors $X,Y,Z,W\in T_pM$ spanning an antiholomorphic linear 
subspace of $T_pM$. Consider the 
function $g_{_{XY}}:D\to \real$, given by $g_{_{XY}}(U)=\lf<\bar R(U,JU)X,Y\rg>$, where $D$ is 
the set of unit vectors in $T_pM$  orthogonal to the vectors $X, JX, Y, JY$. By Proposition \ref{ZJZproperty}, 
the function $g_{_{XY}}$ is constant on 
its domain, hence we have that 
$$0=(dg_{_{XY}})_{_Z}(JW)=2\lf<\bar R(W,Z)X,Y\rg>,$$
and thus the proof is complete.
\end{proof}
\begin{remark} One could ask if manifolds satisfying the antisymmetry property
$$\lf<\bar R(X,JX)Y,JX\rg>=-\lf<\bar R(Y,JY)X,JY\rg>,$$ 
for any orthonormal tangent vectors $X, JX, Y, JY$ in $T_pM$ and any $p\in M$, could give us another 
interesting class of K\"ahler manifolds. However, this class agrees with the manifolds of constant holomorphic 
sectional curvature when the real dimension is at least $6$. Indeed, following the same idea as in \cite{gm} 
we set $X=\frac{A+B}{\sqrt{2}}$, $Y=\frac{A-B}{\sqrt{2}}$ and apply the above antisymmetry obtaining that $\lf<\bar R(X,JX)JX,X\rg>=\lf<\bar R(Y,JY)JY,Y\rg>$ for 
any orthonormal vectors $X,JX, Y,JY$, which implies that the holomorphic sectional curvature is constant under this
dimension condition.
\end{remark}

\section{Proof of Theorem \ref{complextheorem}}
Consider a K\"ahler manifold $M$ of real dimension $2m$, $m\ge 2$. 
If we consider the function $h:S^{m-1}\to \Bbb R$ given by 
$h(X)=\lf<\bar R(X, JX) JX, X\rg>$, where $S^{m-1}$ is the unit sphere on $T_pM$, the 
derivative $dh_{_X}Y=-4\lf<\bar R(X, JX) Y, JX\rg>$, hence $h$ is constant if, and only if, 
$\lf<\bar R(X, JX)Y, JX\rg>=0$ for all $Y$ orthogonal to $X$. 
Since $T_X(S^{m-1})=\mbox{span}(JX)\oplus (\mbox{span}(X, JX))^\perp$, by using 
the K\"ahlerian analogue of Schur's Lemma it 
is easy to obtain the following well known result (see \cite{gm}).
\begin{proposition} \label{equal0} Let $M$ be a K\"ahler manifold of real dimension $2m, m\ge 2$. Then 
$M$ has constant holomorphic sectional curvature if, and only if, for any $p\in M$ and any orthonormal vectors $X, JX, Y, JY\in T_pM$ 
it holds that
\begin{equation} \lf<\bar R(X,JX)Y,JX\rg>=0.
\end{equation}
\end{proposition}
By replacing $Y$ by $JY$ in Proposition \ref{equal0} one obtains the following 
\begin{corollary} \label{drop} $M$ has constant holomorphic 
sectional curvature 
if, and only if, it holds that 
\begin{equation} \lf<\bar R(X,JX)Y,X\rg>=0,
\end{equation}
for any orthonormal vectors $X,JX,Y,JY\in T_pM$ and any $p\in M$.
\end{corollary}
\begin{proof}[Proof of Theorem {\rm\ref{complextheorem}}] Let $M$ be a K\"ahler manifold of 
dimension $2m, m\ge 2$. We first assume that $M$ satisfies the axiom of special holomorphic 
 $r$-submanifolds for some $1\le r\le m-1$. Fix $p\in M$ and orthonormal vectors $X, JX, Y, JY\in T_pM$. By Proposition
\ref{equal0} it suffices to show that $\lf<\bar R(X,JX)Y,JX\rg>=0$.

If $r=1$  by hypothesis there exists a special surface $S$ containing $p$ such that $T_pS={\rm span}(X, JX)$. By using the definition 
of a special surface we have that 
$$0=\lf<\bar R(X,X)X,Y\rg>+\lf<\bar R(JX,X)JX,Y\rg>=\lf<\bar R(X,JX)Y,JX\rg>,$$
hence the holomorphic sectional curvature of $M$ is constant.

Now we assume that $r\ge 2$, hence $m\ge 3$. There exist orthonormal vectors 
$$X_1=X,\, JX_1,\,X_2,\,JX_2,\cdots, X_r,\,JX_r\in ({\rm span}(Y, JY))^\perp.$$
Set $W={\rm span}(X_1,JX_1,\cdots,X_r,JX_r)$. There exists a special submanifold $S_1$ containing $p$ such that 
$T_p(S_1)=W$. Then we have that
\begin{equation}\label{complspecial}\lf<\bar R(JX, X)JX, Y\rg>+ \sum_{i=2}^r\Bigl(\lf<\bar R(X_i,X)X_i,Y\rg>+\lf<\bar R(JX_i,X)JX_i,Y\rg>\Bigr)=0.
\end{equation}
Now set $\Omega={\rm span}(X_2,JX_2,\cdots,X_r,JX_r,Y,JY)$. There exists a special submanifold $S_2$ containing $p$ such that $T_p(S_2)=\Omega$. Then we have that
\begin{equation}\label{complspecial2} \lf<\bar R(JY,Y)JY,X\rg>+\sum_{i=2}^r\lf(\Bigl<\bar R(X_i,Y)X_i,X\rg>+\lf<\bar R(JX_i,Y)JX_i,X\rg>\Bigr)=0.
\end{equation}
From (\ref{complspecial}) and (\ref{complspecial2}) we see that $\lf<\bar R(X,JX)Y,JX\rg>=\lf<\bar R(Y,JY)X,JY\rg>$, 
hence $M$ is a $XY$-manifold. By Proposition  \ref{ZJZproperty} we have that 
$$\lf<\bar R(JX, X)JX, Y\rg>=2\lf(\lf<\bar R(X_i,X)X_i,Y)\rg>+\lf<\bar R(JX_i,X)JX_i,Y\rg>\rg),$$
for all $2\le i\le r$. This fact together with (\ref{complspecial}) implies that $\lf<\bar R(JX, X)JX, Y\rg>=0$,
hence $M$ has constant holomorphic sectional curvature.

Now we will assume that $M$ satisfies the axiom of special antiholomorphic $r$-submanifolds for some $2\le r\le m$. Again 
we fix $p\in M$ and orthonormal vectors $X, JX, Y, JY\in T_pM$. We consider orthonormal vectors
$X_1=X, X_2=Y,\cdots, X_r$ spanning an antiholomorphic subspace $W$ of $T_pM$. There exists a special submanifold 
$S_1$ containing $p$  satisfying $T_p(S_1)=W$. Thus we have that:
\begin{equation}\label{anti1}\lf<\bar R(X,Y)X,JY\rg>+\sum_{i=2}^r\lf<\bar R(X_i,Y)X_i,JY\rg>=0,
\end{equation}
and
\begin{equation}\label{anti3}\lf<\bar R(X,Y)X,JX\rg>+\sum_{i=2}^r\lf<\bar R(X_i,Y)X_i,JX\rg>=0.
\end{equation}

Now set $\Omega={\rm span}(JX, X_2,\cdots,X_r)$. Let $S_2$ be a special submanifold satisfying 
$T_p(S_2)=\Omega$. We have:
\begin{equation}\label{anti2}\lf<\bar R(JX,Y)JX,JY\rg>+\sum_{i=2}^r\lf<\bar R(X_i,Y)X_i,JY\rg>=0.
\end{equation}
From (\ref{anti1}) and (\ref{anti2}) it follows that 
$$\lf<\bar R(X,Y)X,JY\rg>=\lf<\bar R(JX,Y)JX,JY\rg>=-\lf<\bar R(X,Y)X,JY\rg>,$$
hence $\lf<\bar R(X,Y)X,JY\rg>=0$. Since $X, Y$ are arbitrary we conclude by Proposition \ref{xyequivalence} that $M$ is 
a $XY$-manifold. 

If $m=2$ then $r=2$ and it follows immediately 
from (\ref{anti3}) that $\lf<\bar R(X,Y)X, JX\rg>=0$, hence Corollary \ref{drop} implies that 
$M$ has constant holomorphic sectional curvature. 
Now assume that $m\ge 3$. We apply Propositions \ref{ZJZproperty} and \ref{xyequivalence} to each term $\lf<\bar R(X_i,Y)X_i,JX\rg>$ in (\ref{anti3}), 
with $3\le i\le m$, obtaining that 
\begin{eqnarray*}4\lf<\bar R(X_i, Y)X_i,JX\rg>&=&\lf<\bar R(Y,JY)JX, JY\rg>\\
&=&\lf<\bar R(Y,JY)X,Y\rg>\\
&=&-\lf<\bar R(X,JX)Y,X\rg>\\
&=&\lf<\bar R(X,Y)X, JX\rg>.
\end{eqnarray*}
By this equality we obtain from (\ref{anti3}) that $\lf<\bar R(X,Y)X, JX\rg>=0$, hence we conclude again that $M$ has constant holomorphic sectional curvature. 

To conclude the proof of Theorem \ref{complextheorem} we need to prove that complex and totally real 
immersions in a K\"ahler manifold of constant holomorphic sectional curvature are very 
special. In fact, if $M$ has constant holomorphic sectional curvature and $X,JX,Y,JY$ are orthonormal vectors 
in $T_pM$ for some point $p\in M$, we use together Propositions \ref{xyequivalence},
\ref{equal0} and Corollary \ref{drop}, obtaining that
\begin{equation}\label{0xyprop}0=\lf<\bar R(X,JX)Y,JX\rg>=\lf<\bar R(X,JX)Y,X\rg>=\lf<\bar R(X,Y)X,JY\rg>.
\end{equation}
If furthermore $Z\in ({\rm span}(X,JX,Y,JY))^\perp$ we use Propositions \ref{ZJZproperty} and
\ref{equal0}, as well as Corollary \ref{drop} obtaining that  
\begin{equation}\label{0Z}0=\lf<\bar R(Z,X)Z,Y\rg>=\lf<\bar R(JZ,X)JZ,Y\rg>=\lf<\bar R(Z,JX)Z,Y\rg>.
\end{equation} 
Let $f:S\to M$ be a complex immersion in a K\"ahler manifold of constant holomorphic sectional curvature.  
Take a point 
$p\in S$ and fix $Z\in (T_pS)^\perp$ and $X\in T_pS$.  We have by (\ref{0xyprop}) 
that $\lf<\bar R(X,JX)X,Z\rg>=\lf<\bar R(JX,X)JX,Z\rg>=
0$. Thus by linearity we see that $S$ is very special if its real dimension is $2$.  
If the real dimension of $S$ is at least $4$, we fix $Y\in T_pS$ orthogonal to $X, JX$. Since $f$ is 
a complex immersion we have that  $Z\in ({\rm span}(X,JX,Y,JY))^\perp$, hence we obtain from  (\ref{0Z}) that  
$$\lf<\bar R(X,Y)X,Z\rg>=\lf<\bar R(JX,Y)JX,Z\rg>=\lf<\bar R(Y,X)Y,Z\rg>=\lf<\bar R(Y,JX)Y,Z\rg>=0.$$ By the linearity properties of the curvature tensor in a K\"ahler manifold we conclude that $f$ is very special.

Now let $f:S\to M$ be a totally real immersion, where $M$ has constant holomorphic sectional curvature. 
Take $p\in S$ and orthonormal vectors $X,Y\in T_pS$.  
By (\ref{0xyprop}) we have that $\lf<\bar R(X, Y)X, JX\rg>=0$ and $\lf<\bar R(X, Y)X, JY\rg>=0$. 
If  $Z\in (T_pS)^\perp$ is orthogonal to $X,JX,Y,JY$ we obtain from (\ref{0Z}) that $\lf<\bar R(X, Y)X, Z\rg>=0$. 
By interchanging $X$ and $Y$ we obtain similar equations. Thus 
we may use the linearity properties of the curvature tensor in a K\"ahler manifold to conclude that $f$ is very special.

Theorem \ref{complextheorem} is proved. 
\end{proof}

\end{document}